\documentclass[]{interact}

\usepackage{epstopdf}
\usepackage[caption=false]{subfig}
\usepackage{xcolor}
\usepackage{amsmath, amssymb, graphicx}
\usepackage{caption}
\usepackage{amsmath}
\usepackage{amsmath,cases}
\usepackage{mathtools}
\usepackage{xcolor}
\usepackage[numbers,sort&compress]{natbib}
\usepackage{mathtools}
\bibpunct[, ]{[}{]}{,}{n}{,}{,}
\renewcommand\bibfont{\fontsize{10}{12}\selectfont}
\makeatletter
\def\NAT@def@citea{\def\@citea{\NAT@separator}}
\makeatother

\theoremstyle{plain}
\newtheorem{theorem}{Theorem}[section]
\newtheorem{lemma}[theorem]{Lemma}

\newtheorem{proposition}[theorem]{Proposition}

\theoremstyle{definition}
\newtheorem{definition}[theorem]{Definition}
\newtheorem{example}[theorem]{Example}

\theoremstyle{remark}
\newtheorem{remark}{Remark}

\begin{document}

\articletype{Research Paper}

\title{Monotone operator methods for a class of nonlocal multi-phase variable exponent problems}

\author{\name{Mustafa Avci\thanks{CONTACT M.~Avci. Email:  mavci@athabascau.ca (primary) \& avcixmustafa@gmail.com}}
\affil{Faculty of Science and Technology, Applied Mathematics, Athabasca University, AB, Canada}}

\maketitle

\begin{abstract}
In this paper, we study a class of nonlocal multi-phase variable exponent problems within the framework of a newly introduced Musielak-Orlicz Sobolev space. We consider two problems, each distinguished by the type of nonlinearity it includes. To establish the existence of at least one nontrivial solution for each problem, we employ two different monotone operator methods.
\end{abstract}


\begin{amscode}
35A01; 35A15; 35D30; 35J66; 35J75
\end{amscode}

\section{Introduction}

In this article, we study the following nonlocal multi-phase variable exponent problem

\begin{equation}\label{e1.1}
\begin{cases}
\begin{array}{rlll}
-\mathcal{M}(\varrho_{\mathcal{T}}(u))\mathrm{div}(|\nabla u|^{p(x)-2}\nabla u+\mu_1(x)|\nabla u|^{q(x)-2}\nabla u+\mu_2(x)|\nabla u|^{r(x)-2}\nabla u)&=f \text{ in }\Omega, \\
u&=0  \text{ on }\partial \Omega, \tag{$\mathcal{P}$}
\end{array}
\end{cases}
\end{equation}
with
\begin{align*}
\varrho_{\mathcal{T}}(u):=\int_\Omega\left(\frac{|\nabla u|^{p(x)}}{p(x)}+\mu_1(x)\frac{|\nabla u|^{q(x)}}{q(x)}+\mu_2(x)\frac{|\nabla u|^{r(x)}}{r(x)}\right)dx,
\end{align*}
where $\Omega$ is a bounded domain in $\mathbb{R}^N$ $(N\geq2)$ with Lipschitz boundary; $f \in W_0^{1,\mathcal{T}}(\Omega)^{*}$; $\mathcal{M}$ is a $C^1$-continuous nondecreasing function; $p,q,r \in C_+(\overline{\Omega })$ with $1<p(x)<q(x)<r(x)$; and $0\leq \mu_1(\cdot),\mu_2(\cdot)\in L^\infty(\Omega)$.\\

The operator
\begin{equation}\label{e1.2a}
\mathrm{div}(|\nabla u|^{p(x)-2}\nabla u+\mu_1(x)|\nabla u|^{q(x)-2}\nabla u+\mu_2(x)|\nabla u|^{r(x)-2}\nabla u)
\end{equation}
governs anisotropic and heterogeneous diffusion and is associated with the energy functional
\begin{equation}\label{e1.2b}
u\to \int_\Omega\left(\frac{|\nabla u|^{p(x)}}{p(x)}+\mu_1(x)\frac{|\nabla u|^{q(x)}}{q(x)}+\mu_2(x)\frac{|\nabla u|^{r(x)}}{r(x)}\right)dx,\,\ u \in W_0^{1,\mathcal{T}}(\Omega).
\end{equation}
This operator is referred to as a "multi-phase" operator because it encapsulates three distinct types of elliptic behavior within a unified framework. Such a structure allows the model to describe phenomena where materials or processes exhibit varying properties in different regions—for instance, materials that are harder in some areas and softer in others.\\

The energy functional given in (\ref{e1.2b}) was first introduced in \cite{de2019regularity} for constant exponents, where the authors established regularity results for multi-phase variational problems. Later, \cite{vetro2024priori} investigated Dirichlet problems driven by multi-phase operators with variable exponents, providing a priori upper bounds for weak solutions. More recently, \cite{dai2024regularity} examined multi-phase operators with variable exponents, analyzing the associated Musielak-Orlicz Sobolev spaces, extending Sobolev embedding results, and establishing essential regularity properties. Additionally, they demonstrated existence and uniqueness results for Dirichlet problems involving gradient-dependent nonlinearity and derived local regularity estimates.\\

To provide historical context, we also discuss the development of double-phase operators associated with the energy functional
\begin{equation}\label{e1.2bc}
u\to \int_\Omega\left(\frac{|\nabla u|^{p}}{p}+\mu(x)\frac{|\nabla u|^{q}}{q}\right)dx.
\end{equation}
This type of functional was introduced in \cite{zhikov1987averaging}, and since then, numerous studies have explored its properties and applications (see, e.g., \cite{baroni2015harnack,baroni2018regularity,colombo2015bounded,colombo2015regularity,marcellini1991regularity,marcellini1989regularity}). The significance of this model extends across multiple disciplines, underscoring its broad applicability.\\

While preparing this article, we could only find the paper \cite{vetro2025multiplicity} where a Kirchhoff-type (i.e. nonlocal) problem involving a multi-phase operator with variable exponents is studied. In this paper, the author investigate a Kirchhoff-type problem involving a multi-phase operator with three variable exponents. The problem features a right-hand side comprising a Carathéodory perturbation, which is defined locally, along with a Kirchhoff term. By employing a generalized version of the symmetric mountain pass theorem and leveraging recent a priori upper bounds for multi-phase problems, the author establishes the existence of sequence of nontrivial solutions which converges to zero in the corresponding Musielak-Orlicz Sobolev space as well as in $L^{\infty}(\Omega)$.\\

The paper is organised as follows. In Section 2, we first provide some background for the theory of variable Sobolev spaces $W_{0}^{1,p(x)}(\Omega)$ and the Musielak-Orlicz Sobolev space $W_0^{1,\mathcal{H}}(\Omega)$, and then obtain a crucial auxiliary result. In Section 3, we set up the first problem where we work with a general nonlinearity $f \in W_0^{1,\mathcal{T}}(\Omega)^{*}$, and obtain the existence and uniqueness result for (\ref{e1.1}). In Section 4, we study the second problem where we specify the nonlinearity $f$ as $f=f(x,u,\nabla u)$, and obtain an existence result for (\ref{e1.1}).

\section{Mathematical Background and Auxiliary Results}

We start with some basic concepts of variable Lebesgue-Sobolev spaces. For more details, and the proof of the following propositions, we refer the reader to \cite{cruz2013variable,diening2011lebesgue,edmunds2000sobolev,fan2001spaces,radulescu2015partial}.
\begin{equation*}
C_{+}\left( \overline{\Omega }\right) =\left\{h\in C\left( \overline{\Omega }\right) ,\text{ } h\left( x\right) >1 \text{ for all\ }x\in
\overline{\Omega }\right\} .
\end{equation*}
For $h\in C_{+}( \overline{\Omega }) $ denote
\begin{equation*}
h^{-}:=\underset{x\in \overline{\Omega }}{\min }h(x) \leq h(x) \leq h^{+}:=\underset{x\in \overline{\Omega }}{\max}h(x) <\infty .
\end{equation*}
For any $h\in C_{+}\left(\overline{\Omega}\right) $, we define \textit{the
variable exponent Lebesgue space} by
\begin{equation*}
L^{h(x)}(\Omega) =\left\{ u\mid u:\Omega\rightarrow\mathbb{R}\text{ is measurable},\int_{\Omega }|u(x)|^{h(x)}dx<\infty \right\}.
\end{equation*}
Then, $L^{h(x)}(\Omega)$ endowed with the norm
\begin{equation*}
|u|_{h(x)}=\inf \left\{ \lambda>0:\int_{\Omega }\left\vert \frac{u(x)}{\lambda }
\right\vert^{h(x)}dx\leq 1\right\} ,
\end{equation*}
becomes a Banach space.
The convex functional $\rho :L^{h(x)}(\Omega) \rightarrow\mathbb{R}$ defined by
\begin{equation*}
\rho(u) =\int_{\Omega }|u(x)|^{h(x)}dx,
\end{equation*}
is called modular on $L^{h(x)}(\Omega)$.
\begin{proposition}\label{Prop:2.2} If $u,u_{n}\in L^{h(x)}(\Omega)$, we have
\begin{itemize}
\item[$(i)$] $|u|_{h(x)}<1 ( =1;>1) \Leftrightarrow \rho(u) <1 (=1;>1);$
\item[$(ii)$] $|u|_{h(x)}>1 \implies |u|_{h(x)}^{h^{-}}\leq \rho(u) \leq |u|_{h(x)}^{h^{+}}$;\newline
$|u|_{h(x)}\leq1 \implies |u|_{h(x)}^{h^{+}}\leq \rho(u) \leq |u|_{h(x)}^{h^{-}};$
\item[$(iii)$] $\lim\limits_{n\rightarrow \infty }|u_{n}-u|_{h(x)}=0\Leftrightarrow \lim\limits_{n\rightarrow \infty }\rho (u_{n}-u)=0 \Leftrightarrow \lim\limits_{n\rightarrow \infty }\rho (u_{n})=\rho(u)$.
\end{itemize}
\end{proposition}
\begin{proposition}\label{Prop:2.2bb}
Let $h_1(x)$ and $h_2(x)$ be measurable functions such that $h_1\in L^{\infty}(\Omega )$ and $1\leq h_1(x)h_2(x)\leq \infty$ for a.e. $x\in \Omega$. Let $u\in L^{h_2(x)}(\Omega ),~u\neq 0$. Then
\begin{itemize}
\item[$(i)$] $\left\vert u\right\vert _{h_1(x)h_2(x)}\leq 1\text{\ }\Longrightarrow
\left\vert u\right\vert_{h_1(x)h_2(x)}^{h_1^{+}}\leq \left\vert \left\vert
u\right\vert ^{h_1(x)}\right\vert_{h_2(x)}\leq \left\vert
u\right\vert _{h_1(x)h_2(x)}^{h_1^{-}}$
\item[$(ii)$] $\left\vert u\right\vert_{h_1(x)h_2(x)}\geq 1\ \Longrightarrow \left\vert
u\right\vert_{h_1(x)h_2(x)}^{h_1^{-}}\leq \left\vert \left\vert u\right\vert^{h_1(x)}\right\vert_{h_2(x)}\leq \left\vert u\right\vert_{h_1(x)h_2(x)}^{h_1^{+}}$
\item[$(iii)$] In particular, if $h_1(x)=h$ is constant then
\begin{equation*}
\left\vert \left\vert u\right\vert^{h}\right\vert_{h_2(x)}=\left\vert u\right\vert _{hh_2(x)}^{h}.
\end{equation*}
\end{itemize}
\end{proposition}

The variable exponent Sobolev space $W^{1,h(x)}( \Omega)$ is defined by
\begin{equation*}
W^{1,h(x)}( \Omega) =\{u\in L^{p(x) }(\Omega) : |\nabla u| \in L^{h(x)}(\Omega)\},
\end{equation*}
with the norm
\begin{equation*}
\|u\|_{1,h(x)}:=|u|_{h(x)}+|\nabla u|_{h(x)},
\end{equation*}
for all $u\in W^{1,h(x)}(\Omega)$.\\
\begin{proposition}\label{Prop:2.4} If $1<h^{-}\leq h^{+}<\infty $, then the spaces
$L^{h(x)}(\Omega)$ and $W^{1,h(x)}(\Omega)$ are separable and reflexive Banach spaces.
\end{proposition}
The space $W_{0}^{1,h(x)}(\Omega)$ is defined as
$\overline{C_{0}^{\infty }(\Omega )}^{\|\cdot\|_{1,h(x)}}=W_{0}^{1,h(x)}(\Omega)$, and hence, it is the smallest closed set that contains $C_{0}^{\infty }(\Omega )$. Therefore, $W_{0}^{1,h(x)}(\Omega)$ is also a separable and reflexive Banach space due to the inclusion $W_{0}^{1,h(x)}(\Omega) \subset W^{1,h(x)}(\Omega)$. \\
Note that as  a consequence of Poincar\'{e} inequality, $\|u\|_{1,h(x)}$ and $|\nabla u|_{h(x)}$ are equivalent
norms on $W_{0}^{1,h(x)}(\Omega)$. Therefore, for any $u\in W_{0}^{1,h(x)}(\Omega)$ we can define an equivalent norm $\|u\|$ such that
\begin{equation*}
\|u\| :=|\nabla u|_{h(x)}.
\end{equation*}
\begin{proposition}\label{Prop:2.5} Let $m\in C(\overline{\Omega })$. If $1\leq m(x)<h^{\ast }(x)$ for all $x\in
\overline{\Omega}$, then the embeddings $W^{1,h(x)}(\Omega) \hookrightarrow L^{m(x)}(\Omega)$ and $W_0^{1,h(x)}(\Omega) \hookrightarrow L^{m(x)}(\Omega)$  are compact and continuous, where
$h^{\ast}(x) =\left\{\begin{array}{cc}
\frac{Nh(x) }{N-h(x)} & \text{if }h(x)<N, \\
+\infty & \text{if }h(x) \geq N.
\end{array}
\right. $
\end{proposition}
In the sequel, we introduce the multi-phase operator, the Musielak–Orlicz space, and the Musielak–Orlicz Sobolev space, respectively.\\
We make the following assumptions.
\begin{itemize}
\item[$(H_1)$] $p,q,r,s\in C_+(\overline{\Omega})$ with $p(x)<N$; $1<p^-\leq p(x)<q(x)<r(x)<s(x)<p^*(x)$; $s^+<p^*(x)$.
\item[$(H_2)$] $\mu_1(\cdot),\mu_2(\cdot)\in L^\infty(\Omega)$ such that $\mu_1(x)\geq 0$ and $\mu_2(x)\geq 0$ for all $x\in
\overline{\Omega }$.
\end{itemize}
Under the assumptions $(H_1)$ and $(H_2)$, we define the nonlinear function $\mathcal{T}:\Omega\times [0,\infty]\to [0,\infty]$, i.e. the \textit{multi-phase operator}, by
\[
\mathcal{T}(x,t)=t^{p(x)}+\mu_1(x)t^{q(x)}+\mu_2(x)t^{r(x)}\ \text{for all}\ (x,t)\in \Omega\times [0,\infty].
\]
Then the corresponding modular $\rho_\mathcal{T}(\cdot)$ is given by
\[
\displaystyle\rho_\mathcal{T}(u):=\int_\Omega\mathcal{T}(x,|u|)dx=
\int_\Omega\left(|u(x)|^{p(x)}+\mu_1(x)|u(x)|^{q(x)}+\mu_2(x)|u(x)|^{r(x)}\right)dx.
\]
The \textit{Musielak-Orlicz space} $L^{\mathcal{T}}(\Omega)$, is defined by
\[
L^{\mathcal{T}}(\Omega)=\left\{u:\Omega\to \mathbb{R}\,\, \text{measurable};\,\, \rho_{\mathcal{T}}(u)<+\infty\right\},
\]
endowed with the Luxemburg norm
\[
\|u\|_{\mathcal{T}}:=\inf\left\{\zeta>0: \rho_{\mathcal{T}}\left(\frac{u}{\zeta}\right)\leq 1\right\}.
\]
Analogous to Proposition \ref{Prop:2.2}, there are similar relationship between the modular $\rho_{\mathcal{T}}(\cdot)$ and the norm $\|\cdot\|_{\mathcal{T}}$, see \cite[Proposition 3.2]{dai2024regularity} for a detailed proof.

\begin{proposition}\label{Prop:2.2a}
Assume $(H_1)$ hold, and $u\in L^{\mathcal{H}}(\Omega)$. Then
\begin{itemize}
\item[$(i)$] If $u\neq 0$, then $\|u\|_{\mathcal{T}}=\zeta\Leftrightarrow \rho_{\mathcal{T}}(\frac{u}{\zeta})=1$,
\item[$(ii)$] $\|u\|_{\mathcal{T}}<1\ (\text{resp.}\ >1, =1)\Leftrightarrow \rho_{\mathcal{T}}(\frac{u}{\zeta})<1\ (\text{resp.}\ >1, =1)$,
\item[$(iii)$] If $\|u\|_{\mathcal{T}}<1\Rightarrow \|u\|_{\mathcal{T}}^{r^+}\leq \rho_{\mathcal{T}}(u)\leq \|u\|_{\mathcal{T}}^{p^-}$,
\item[$(iv)$]If $\|u\|_{\mathcal{T}}>1\Rightarrow \|u\|_{\mathcal{T}}^{p^-}\leq \rho_{\mathcal{T}}(u)\leq \|u\|_{\mathcal{T}}^{r^+}$,
\item[$(v)$] $\|u\|_{\mathcal{T}}\to 0\Leftrightarrow \rho_{\mathcal{T}}(u)\to 0$,
\item[$(vi)$]$\|u\|_{\mathcal{T}}\to +\infty\Leftrightarrow \rho_{\mathcal{T}}(u)\to +\infty$,
\item[$(vii)$] $\|u\|_{\mathcal{T}}\to 1\Leftrightarrow \rho_{\mathcal{T}}(u)\to 1$,
\item[$(viii)$] If $u_n\to u$ in $L^{\mathcal{T}}(\Omega)$, then $\rho_{\mathcal{T}}(u_n)\to\rho_{\mathcal{T}}(u)$.
\end{itemize}
\end{proposition}
The \textit{Musielak-Orlicz Sobolev space} $W^{1,\mathcal{T}}(\Omega)$ is defined by
\[
W^{1,\mathcal{T}}(\Omega)=\left\{u\in L^{\mathcal{T}}(\Omega):
|\nabla u|\in L^{\mathcal{H}}(\Omega)\right\},
\]
and equipped with the norm
\[
\|u\|_{1,\mathcal{T}}:=\|\nabla u\|_{\mathcal{T}}+\|u\|_{\mathcal{T}},
\]
where $\|\nabla u\|_{\mathcal{T}}=\|\,|\nabla u|\,\|_{\mathcal{T}}$.\\
The space $W_0^{1,\mathcal{T}}(\Omega)$ is defined as $\overline{C_{0}^{\infty }(\Omega )}^{\|\cdot\|_{1,\mathcal{T}}}=W_0^{1,\mathcal{T}}(\Omega)$. Notice that $L^{\mathcal{T}}(\Omega), W^{1,\mathcal{T}}(\Omega)$ and $W_0^{1,\mathcal{T}}(\Omega)$ are uniformly convex and reflexive Banach spaces, and the following embeddings hold \cite[Propositions 3.1, 3.3]{dai2024regularity}.
\begin{proposition}\label{Prop:2.7a}
Let $(H_1)$ and $(H_2)$ be satisfied. Then the following embeddings hold:
\begin{itemize}
\item[$(i)$] $L^{\mathcal{T}}(\Omega)\hookrightarrow L^{h(\cdot)}(\Omega), W^{1,\mathcal{T}}(\Omega)\hookrightarrow W^{1,h(\cdot)}(\Omega)$, $W_0^{1,\mathcal{T}}(\Omega)\hookrightarrow W_0^{1,h(\cdot)}(\Omega)$ are continuous for all $h\in C(\overline{\Omega})$ with $1\leq h(x)\leq p(x)$ for all $x\in \overline{\Omega}$.
\item[$(ii)$] $W^{1,\mathcal{T}}(\Omega)\hookrightarrow L^{h(\cdot)}(\Omega)$ and $W_0^{1,\mathcal{T}}(\Omega)\hookrightarrow L^{h(\cdot)}(\Omega)$ are compact for all $h\in C(\overline{\Omega})$ with $1\leq h(x)< p^*(x)$ for all $x\in \overline{\Omega}$.
\end{itemize}
\end{proposition}
As a conclusion of Proposition \ref{Prop:2.7a}:\\
We have the continuous embedding $W_0^{1,\mathcal{T}}(\Omega)\hookrightarrow L^{h(\cdot)}(\Omega)$, and there is a constant $c_{\mathcal{T}}$ such that
\[
\|u\|_{h(\cdot)}\leq c_{\mathcal{T}}\|u\|_{1,\mathcal{T},0}.
\]
As well, $W_0^{1,\mathcal{H}}(\Omega)$ is compactly embedded in $L^{\mathcal{H}}(\Omega)$.\\
Thus,
$W_0^{1,\mathcal{T}}(\Omega)$ can be equipped with the equivalent norm
\[
\|u\|_{1,\mathcal{T},0}:=\|\nabla u\|_{\mathcal{T}}.
\]
We lastly introduce the seminormed spaces
\begin{equation*}
L^{q(\cdot)}_{\mu_1}(\Omega)=\left\{u:\Omega\to \mathbb{R}\,\, \text{measurable};\,\, \int_\Omega\mu_1(x)|u|^{q(x)}dx<+\infty\right\},
\end{equation*}
and
\begin{equation*}
L^{r(\cdot)}_{\mu_2}(\Omega)=\left\{u:\Omega\to \mathbb{R}\,\, \text{measurable};\,\, \int_\Omega\mu_2(x)|u|^{r(x)}dx<+\infty\right\},
\end{equation*}
which are endow with the seminorms
\begin{equation*}
|u|_{q(\cdot),\mu_1}=\inf\left\{\varsigma_1>0: \int_\Omega\mu_1(x)\left(\frac{|u|}{\varsigma_1}\right)^{q(x)}dx \leq 1 \right\},
\end{equation*}
and
\begin{equation*}
|u|_{r(\cdot),\mu_2}=\inf\left\{\varsigma_2>0: \int_\Omega\mu_2(x)\left(\frac{|u|}{\varsigma_2}\right)^{r(x)}dx \leq 1 \right\},
\end{equation*}
respectively. We have $L^{\mathcal{T}}(\Omega)\hookrightarrow L^{q(\cdot)}_{\mu_1}(\Omega)$ and $L^{\mathcal{T}}(\Omega)\hookrightarrow L^{r(\cdot)}_{\mu_2}(\Omega)$ continuously \cite[Proposition 3.3]{dai2024regularity}.
\begin{proposition}\label{Prop:2.7}
For the convex functional
$$
\varrho_{\mathcal{T}}(u)=\int_\Omega\left(\frac{|\nabla u|^{p(x)}}{p(x)}+\mu_1(x)\frac{|\nabla u|^{q(x)}}{q(x)}+\mu_2(x)\frac{|\nabla u|^{r(x)}}{r(x)}\right)dx,
$$
we have the following  \cite{dai2024regularity}:
\begin{itemize}
\item[$(i)$] $\varrho_{\mathcal{T}} \in C^{1}(W_0^{1,\mathcal{T}}(\Omega),\mathbb{R})$ with the derivative
$$
\langle\varrho^{\prime}_{\mathcal{T}}(u),\varphi\rangle=\int_{\Omega}(|\nabla u|^{p(x)-2}\nabla u+\mu_1(x)|\nabla u|^{q(x)-2}\nabla u+\mu_2(x)|\nabla u|^{r(x)-2}\nabla u)\cdot\nabla \varphi dx,
$$
for all  $u, \varphi \in W_0^{1,\mathcal{T}}(\Omega)$, where $\langle \cdot, \cdot\rangle$ is the dual pairing between $W_0^{1,\mathcal{T}}(\Omega)$ and its dual $W_0^{1,\mathcal{T}}(\Omega)^{*}$;
\item[$(ii)$] $\varrho^{\prime}_{\mathcal{T}}$ satisfies the $(S_{+})$-property, i.e.
\begin{equation}\label{e2.7ab}
u_{n} \rightharpoonup u \text{ in } W_0^{1,\mathcal{T}}(\Omega)
\end{equation}
and
\begin{equation}\label{e2.8ab}
\limsup_{n\rightarrow\infty}\langle \varrho^{\prime}_{\mathcal{T}}(u_{n}),u_{n}-u\rangle\leq 0
\end{equation}
imply
\begin{equation}\label{e2.9ab}
u_{n} \to u \text{ in } W_0^{1,\mathcal{T}}(\Omega).
\end{equation}

\end{itemize}
\end{proposition}
\begin{remark}\label{Rem:3.4ab}
Notice that by Propositions \ref{Prop:2.2a} and the equivalency of the norms $\|u\|_{1,\mathcal{T},0}$ and $\|\nabla u\|_{\mathcal{T}}$, we have the relations:
\begin{equation}\label{e2.de}
\frac{1}{r^+}\|u\|^{p^-}_{1,\mathcal{T},0}\leq \frac{1}{r^+}\rho_{\mathcal{T}}(\nabla u)\leq \varrho_{\mathcal{T}}(u) \leq \frac{1}{p^-}\rho_{\mathcal{T}}(\nabla u)\leq \frac{1}{p^-}\|u\|^{r^+}_{1,\mathcal{T},0} \text{ if } \|u\|_{1,\mathcal{T},0}>1,
\end{equation}
\begin{equation}\label{e2.df}
\frac{1}{r^+}\|u\|^{r^+}_{1,\mathcal{T},0}\leq \frac{1}{r^+}\rho_{\mathcal{T}}(\nabla u)\leq \varrho_{\mathcal{T}}(u) \leq \frac{1}{p^-}\rho_{\mathcal{T}}(\nabla u)\leq \frac{1}{p^-}\|u\|^{p^-}_{1,\mathcal{T},0} \text{ if } \|u\|_{1,\mathcal{T},0}\leq 1.
\end{equation}
\begin{equation}\label{e3.33a}
0<\langle\varrho^{\prime}_{\mathcal{T}}(u),u\rangle=\rho_{\mathcal{T}}(\nabla u),
\end{equation}
\begin{equation}\label{e3.33c}
0<p^- \rho_{\mathcal{T}}(\nabla u)\leq \langle \rho^{\prime}_{\mathcal{T}}(\nabla u), \nabla u\rangle \leq r^+ \rho_{\mathcal{T}}(\nabla u).
\end{equation}
\end{remark}
The following result is obtained by the author in his recently submitted paper, which is still under review. However, since it plays a crucial part to obtain the main regularity results of the present paper, we provide its proof for the convenience of the reader.
\begin{proposition} \label{Prop:2.8a}
Let $x,y \in \mathbb{R}^N$ and let $|\cdot|$ be the Euclidean norm in $\mathbb{R}^N$. Then for any $1\leq p<\infty$ and the real parameters $a,b>0$ it holds
\begin{equation}\label{e2.1a1}
|a |x|^{p-2}x-b |y|^{p-2}y| \leq \left(a+|a-b|\right)||x|^{p-2}x-|y|^{p-2}y|+|a-b|.
\end{equation}
\end{proposition}
\begin{proof}
If $a=b$, then there is nothing to do. So, we assume that $a \neq b$.\\
Put
\begin{equation}\label{e2.1a2}
\Lambda(x,y)=\frac{|a |x|^{p-2}x-b |y|^{p-2}y|}{||x|^{p-2}x-|y|^{p-2}y|}.
\end{equation}
Notice that $\Lambda$ is invariant by any orthogonal transformation $T$; that is, $\Lambda(Tx,Ty)=\Lambda(x,y)$ for all $x,y \in \mathbb{R}^N$.
Thus, using this argument and  the homogeneity of $\Lambda$, we can let $x=|x|e_1$ and assume that $x=e_1$. Thus, it is enough to work with the function
\begin{equation}\label{e2.1a5}
\Lambda(e_1,y)=\frac{|a e_1 -b |y|^{p-2}y|}{|e_1-|y|^{p-2}y|}.
\end{equation}
First we get
\begin{align}\label{e2.1a6}
|a e_1 -b |y|^{p-2}y|& =a\bigg|e_1 -\frac{b}{a} |y|^{p-2}y\bigg|=a\bigg|(e_1-|y|^{p-2}y) +\left(1-\frac{b}{a}\right) |y|^{p-2}y\bigg| \nonumber\\
& \leq a|e_1-|y|^{p-2}y|+|a-b|||y|^{p-2}y| \nonumber\\
& \leq a|e_1-|y|^{p-2}y|+|a-b|\left(|e_1-|y|^{p-2}y|+|e_1|\right) \nonumber\\
& \leq a|e_1-|y|^{p-2}y|+|a-b||e_1-|y|^{p-2}y|+|a-b| \nonumber\\
& \leq |e_1-|y|^{p-2}y|\left(a+|a-b|\right)+|a-b|.
\end{align}
Then using this in (\ref{e2.1a5}) we obtain
\begin{align}\label{e2.1a7}
\Lambda(e_1,y)&\leq\frac{|e_1-|y|^{p-2}y|\left(a+|a-b|\right)+|a-b|}{|e_1-|y|^{p-2}y|} \nonumber\\
&\leq a+|a-b|+\frac{|a-b|}{|e_1-|y|^{p-2}y|},
\end{align}
from which (\ref{e2.1a1}) follows.
\end{proof}

\section{The first problem}

\begin{definition}\label{Def:3.1} A function $u\in W_0^{1,\mathcal{T}}(\Omega)$ is called a weak solution to problem (\ref{e1.1}) if for all test function $\varphi \in W_0^{1,\mathcal{T}}(\Omega)$ it holds
\begin{align}\label{e3.1}
\mathcal{M}(\varrho_{\mathcal{T}}(u))\langle\varrho^{\prime}_{\mathcal{T}}(u),\varphi\rangle=\int_{\Omega}f\varphi dx.
\end{align}
\end{definition}
Let us define the functional $\mathcal{G}:W_0^{1,\mathcal{H}}(\Omega)\rightarrow \mathbb{R}$ as
\begin{equation}\label{e3.2}
\mathcal{G}(u):=\mathcal{\widehat{M}}(\varrho_{\mathcal{T}}(u))=\int_{0}^{\varrho_{\mathcal{T}}(u)}\mathcal{M}(s)ds.
\end{equation}
Therefore, we can define the operator $\mathcal{H}:W_0^{1,\mathcal{T}}(\Omega)\rightarrow W_0^{1,\mathcal{T}}(\Omega)^{*}$ by
\begin{equation} \label{e3.4}
\langle \mathcal{H}(u),\varphi\rangle=\langle f,\varphi\rangle,\,\, \mbox{ for all } u,\varphi \in W_0^{1,\mathcal{T}}(\Omega),
\end{equation}
where $\mathcal{H}(u)=\mathcal{G}^{\prime}(u)$. As it is well-know from the theory of monotone operators \cite{zeidler2013nonlinear}, due to (\ref{e3.1}) and (\ref{e3.4}), one way to show that $u\in W_0^{1,\mathcal{T}}(\Omega)$ is a solution to problem (\ref{e1.1}) for all test functions $\varphi \in W_0^{1,\mathcal{T}}(\Omega)$ is to solve the operator equation
\begin{equation} \label{e3.5c}
\mathcal{H}u=f.
\end{equation}

We employ the following well-known result from nonlinear monotone operator theory (see, e.g., \cite{zeidler2013nonlinear} for further details).

\begin{lemma}\label{Lem:3.1} \cite{browder1963nonlinear,minty1963monotonicity}
Let $X$ be a reflexive real Banach space. Let $A:X\rightarrow X^{*}$ be an (nonlinear) operator satisfying the following:
\begin{itemize}
\item[$(i)$] $A$ is coercive.
\item[$(ii)$] $A$ is hemicontinuous; that is, $A$ is directionally weakly continuous, iff the function
$$
\Phi(\theta)=\langle A(u+\theta w),v \rangle
$$
is continuous in $\theta$ on $[0,1]$ for every $u,w,v\in X$.
\item[$(iii)$]
 $A$ is monotone on the space $X$; that is, for all $u,v\in X$ we
have
\begin{equation} \label{e3.7qa}
\langle A\left( u\right)- A\left( v\right),u-v\rangle \geq 0.
\end{equation}
\end{itemize}
Then equation
\begin{equation}\label{e3.7qc}
Au=g
\end{equation}
has at least one nontrivial solution $u\in X$ for every $g\in X^{*}$. If, moreover, the inequality (\ref{e3.7qa}) is strict for all $u, v \in X$ , $u \neq v$, then the equation (\ref{e3.7qc}) has precisely one solution $u\in X$ for every $g\in X^{*}$.
\end{lemma}

The following is the first main result.

\begin{theorem}\label{Thrm:3.1} Assume that the hypotheses $(H_1)-(H_2)$ are satisfied. Additionally, assume that the function $\mathcal{M}$ satisfies the following:
\begin{itemize}
\item[$(M)$] $\mathcal{M}:(0,\infty)\to [m_0, \infty)$ is a $C^1$-continuous nondecreasing function such that
    \begin{equation}\label{e3.a7a}
    m_0\leq \mathcal{M}(t)\leq \kappa t^{\gamma-1},
    \end{equation}
    where $m_0,\kappa,\gamma$ are positive real parameters with $\gamma>1$.
\end{itemize}
Then for given any $f\in W_0^{1,\mathcal{T}}(\Omega)^*$, the operator equation (\ref{e3.5c}) has a unique nontrivial solution $u \in W_0^{1,\mathcal{T}}(\Omega)$ which in turn becomes a nontrivial weak solution to problem (\ref{e1.1}).
\end{theorem}

\begin{lemma}\label{Lem:3.2}
$\mathcal{H}$ is coercive.
\end{lemma}
\begin{proof}
Using $(M)$ and Remark \ref{Rem:3.4ab} it reads
\begin{equation}\label{e3.a1a}
\langle\mathcal{H}(u),u\rangle=\mathcal{M}(\varrho_{\mathcal{T}}(u))\rho_{\mathcal{T}}(\nabla u)\geq \frac{m_0}{r^+}\|u\|_{1,\mathcal{T},0}^{p^{-}}.
\end{equation}
Hence, $\frac{\langle \mathcal{H}(u),u\rangle }{\|u\|_{1,\mathcal{T},0}} \to +\infty$ as $\|u\|_{1,\mathcal{T},0}\to \infty$ since $p^{-}>1$, which implies that $\mathcal{H}$ is coercive.\\
\end{proof}

\begin{lemma}\label{Lem:3.3}
$\mathcal{H}$ is hemicontinuous.
\end{lemma}

\begin{proof}
Next, we show that operator $\mathcal{H}$ is hemicontinuous. Then
\begin{align}\label{e3.7aa}
&|\Phi(\theta_{1})-\Phi(\theta_{2})|
=|\langle \mathcal{H}(u+\theta_{1} v)-\mathcal{H}(u+\theta_{2} v), \varphi \rangle|\nonumber \\
& \leq \int_{\Omega}\bigg|\mathcal{M}_{\theta_1}|\nabla (u+\theta_{1} v)|^{p(x)-2}\nabla (u+\theta_{1} v)-\mathcal{M}_{\theta_2}|\nabla(u+\theta_{2} v)|^{p(x)-2}\nabla (u+\theta_{2} v)\bigg||\nabla \varphi| dx \nonumber \\
&+ \int_{\Omega}\mu_1(x)\bigg|\mathcal{M}_{\theta_1}|\nabla (u+\theta_{1} v)|^{q(x)-2}\nabla (u+\theta_{1} v)-\mathcal{M}_{\theta_2}|\nabla(u+\theta_{2} v)|^{q(x)-2}\nabla (u+\theta_{2} v) \bigg||\nabla \varphi| dx\nonumber \\
&+ \int_{\Omega}\mu_2(x)\bigg|\mathcal{M}_{\theta_1}|\nabla (u+\theta_{1} v)|^{r(x)-2}\nabla (u+\theta_{1} v)-\mathcal{M}_{\theta_2}|\nabla(u+\theta_{2} v)|^{r(x)-2}\nabla (u+\theta_{2} v) \bigg||\nabla \varphi| dx.
\end{align}
where we let $\mathcal{M}_{\theta_1}=\mathcal{M}(\varrho_{\mathcal{T}}(u+\theta_{1} v))$ and $\mathcal{M}_{\theta_2}=\mathcal{M}(\varrho_{\mathcal{T}}(u+\theta_{2} v))$ for the sake of simplicity. Using Proposition \ref{Prop:2.8a}, it reads
\begin{align}\label{e3.7ab}
&|\Phi(\theta_{1})-\Phi(\theta_{2})| \nonumber \\
& \leq \int_{\Omega}\left\{||\nabla (u+\theta_{1} v)|^{p(x)-2}\nabla (u+\theta_{1} v)-|(u+\theta_{2} v)|^{p(x)-2}\nabla (u+\theta_{2} v)|\right.\nonumber\\
&\left.\times (K^{\theta_2}_{\theta_1}+\mathcal{M}_{\theta_1})+K^{\theta_2}_{\theta_1}\right\} |\nabla \varphi| dx \nonumber \\
& + \int_{\Omega}\mu_1(x)\left\{||\nabla (u+\theta_{1} v)|^{q(x)-2}\nabla (u+\theta_{1} v)-|(u+\theta_{2} v)|^{q(x)-2}\nabla (u+\theta_{2} v)| \right.\nonumber\\
&\left.\times (K^{\theta_2}_{\theta_1}+\mathcal{M}_{\theta_1})+K^{\theta_2}_{\theta_1}\right\} |\nabla \varphi| dx \nonumber \\
& + \int_{\Omega}\mu_2(x)\left\{||\nabla (u+\theta_{1} v)|^{r(x)-2}\nabla (u+\theta_{1} v)-|(u+\theta_{2} v)|^{r(x)-2}\nabla (u+\theta_{2} v)|\right.\nonumber\\
\displaybreak
&\left.\times (K^{\theta_2}_{\theta_1}+\mathcal{M}_{\theta_1})+K^{\theta_2}_{\theta_1}\right\} |\nabla \varphi| dx \nonumber \\
& \leq (K^{\theta_2}_{\theta_1}+\mathcal{M}_{\theta_1})\int_{\Omega}||\nabla (u+\theta_{1} v)|^{p(x)-2}\nabla (u+\theta_{1} v)-|\nabla (u+\theta_{1} v)|^{p(x)-2}\nabla (u+\theta_{1} v)| |\nabla v| dx \nonumber \\
& + (K^{\theta_2}_{\theta_1}+\mathcal{M}_{\theta_1})\int_{\Omega}\mu_1(x)||\nabla (u+\theta_{1} v)|^{q(x)-2}\nabla u_n-|\nabla (u+\theta_{1} v)|^{q(x)-2}\nabla (u+\theta_{1} v)||\nabla v| dx \nonumber \\
& + (K^{\theta_2}_{\theta_1}+\mathcal{M}_{\theta_1})\int_{\Omega}\mu_2(x)||\nabla (u+\theta_{1} v)|^{r(x)-2}\nabla u_n-|\nabla (u+\theta_{1} v)|^{r(x)-2}\nabla (u+\theta_{1} v)||\nabla v| dx \nonumber\\ &+3K^{\theta_2}_{\theta_1}\int_{\Omega} |\nabla \varphi| dx,
\end{align}
where $K^{\theta_2}_{\theta_1}=|\mathcal{M}_{\theta_1}-\mathcal{M}_{\theta_2}|$.\\
Recall the following inequality \cite{chipot2009elliptic}: for any $1<m<\infty$, there is a constant $c_m>0$ such that
\begin{equation}\label{e3.6}
(|a|^{m-2}a-|b|^{m-2}b) \leq c_m |a-b|(|a|+|b|)^{m-2},\quad \forall a,b \in \mathbb{R}^{N}.
\end{equation}
Therefore,
\begin{align}\label{e3.6aa}
&|\Phi(\theta_{1})-\Phi(\theta_{2})|\nonumber\\
& \leq 2^{p^+-1}|\theta_{1}-\theta_{2}| (K^{\theta_2}_{\theta_1}+\mathcal{M}_{\theta_1})\int_{\Omega}\left(|\nabla (u+\theta_{1} v)|^{p(x)-2}+|\nabla(u+\theta_{2} v)|^{p(x)-2}\right)|\nabla v||\nabla \varphi| dx \nonumber\\
&+ 2^{q^+-1}|\theta_{1}-\theta_{2}| (K^{\theta_2}_{\theta_1}+\mathcal{M}_{\theta_1})\int_{\Omega}\mu_1(x)\left(|\nabla (u+\theta_{1} v)|^{q(x)-2}+|\nabla(u+\theta_{2} v)|^{q(x)-2}\right)|\nabla v||\nabla \varphi| dx \nonumber\\
&+ 2^{r^+-1}|\theta_{1}-\theta_{2}|(K^{\theta_2}_{\theta_1}+\mathcal{M}_{\theta_1}) \int_{\Omega}\mu_2(x)\left(|\nabla (u+\theta_{1} v)|^{r(x)-2}+|\nabla(u+\theta_{2} v)|^{r(x)-2}\right)|\nabla v||\nabla \varphi| dx\nonumber\\
&+3K^{\theta_2}_{\theta_1}\int_{\Omega} |\nabla \varphi| dx.
\end{align}
Note that since $\theta_{1},\theta_{2} \in [0,1]$, we have $|\nabla (u+\theta_{1} v)|\leq |\nabla u|+|\nabla v|$,  $|\nabla (u+\theta_{2} v)|\leq |\nabla u|+|\nabla v|$, and the fact $|\nabla v|\leq |\nabla u|+|\nabla v|$. Therefore, applying these and letting $|\nabla u|+|\nabla v|=\xi$ in the lines above, and using the H\"{o}lder inequality, Proposition \ref{Prop:2.2bb}, and the necessary embeddings from Proposition \ref{Prop:2.7a} leads to
\begin{align}\label{e3.6ab}
&|\Phi(\theta_{1})-\Phi(\theta_{2})|\nonumber\\
& \leq 2^{p^+}|\theta_{1}-\theta_{2}| (K^{\theta_2}_{\theta_1}+\mathcal{M}_{\theta_1}) \int_{\Omega}|\xi|^{p(x)-2}|\xi||\nabla \varphi| dx\nonumber\\
&+ 2^{q^+}|\theta_{1}-\theta_{2}| (K^{\theta_2}_{\theta_1}+\mathcal{M}_{\theta_1}) \int_{\Omega}\mu_1(x)|\xi|^{q(x)-2}|\xi|\nabla \varphi| dx\nonumber\\
&+ 2^{r^+}|\theta_{1}-\theta_{2}| (K^{\theta_2}_{\theta_1}+\mathcal{M}_{\theta_1}) \int_{\Omega}\mu_2(x)|\xi|^{r(x)-2}|\xi|\nabla \varphi| dx+3K^{\theta_2}_{\theta_1}\int_{\Omega} |\nabla \varphi| dx \nonumber\\
& \leq 2^{r^+}|\theta_{1}-\theta_{2}| (K^{\theta_2}_{\theta_1}+\mathcal{M}_{\theta_1})\nonumber\\
& \times \left(|\xi|_{p(x)}^{p^+-1} |\nabla \varphi|_{p(x)}+|\mu_1|_{\infty}|\xi|_{q(x)}^{q^+-1} |\nabla \varphi|_{q(x)}+|\mu_2|_{\infty}|\xi|_{r(x)}^{r^+-1} |\nabla \varphi|_{r(x)}+3 \right)\nonumber\\
& \leq 2^{r^++1}|\theta_{1}-\theta_{2}| (K^{\theta_2}_{\theta_1}+\mathcal{M}_{\theta_1})\left(\|\xi\|_{1,\mathcal{T},0}^{p^+-1} +|\mu_1|_{\infty}\|\xi\|_{1,\mathcal{T},0}^{q^+-1}+ |\mu_2|_{\infty}\|\xi\|_{1,\mathcal{T},0}^{r^+-1}+3\right) \|\varphi\|_{1,\mathcal{T},0}
\end{align}
Notice that by $(M)$ and Proposition \ref{Prop:2.7}, we have \newline $K^{\theta_2}_{\theta_1}=|\mathcal{M}_{\theta_1}-\mathcal{M}_{\theta_2}| \to 0$, and $\mathcal{M}_{\theta_1} \to \mathcal{M}_{\theta_2} \in [m_0, \infty)$ as $\theta_{1}\rightarrow \theta_{2}$. \newline Therefore, $|\theta_{1}-\theta_{2}| (K^{\theta_2}_{\theta_1}+\mathcal{M}_{\theta_1}) \to 0$ as $\theta_{1}\rightarrow \theta_{2}$. In conclusion, we have
$$
|\Phi(\theta_{1})-\Phi(\theta_{2})|=|\langle \mathcal{H}(u+\theta_{1} v)-\mathcal{H}(u+\theta_{2} v), \varphi \rangle| \rightarrow 0\,\, \text{ as } \theta_{1}\rightarrow \theta_{2},
$$
which implies that $\mathcal{H}$ is hemicontinuous.
\end{proof}

\begin{lemma}\label{Lem:3.4}
$\mathcal{H}$ is strictly monotone.
\end{lemma}

\begin{proof}
Now, we show that $\mathcal{H}$ is strictly monotone. To do so, we argue similarly to \cite{massar2024existence}.
Let $u,v \in W_0^{1,\mathcal{T}}(\Omega)$ with $u \neq v$. Without loss of generality, we can assume that $\varrho_{\mathcal{T}}(u)\geq \varrho_{\mathcal{T}}(v)$. Then, $\mathcal{M}(\varrho_{\mathcal{T}}(u))\geq \mathcal{M}(\varrho_{\mathcal{T}}(v))$ due to $(M)$ and Proposition \ref{Prop:2.7}.\\
Noticing that $\nabla u \cdot \nabla v\leq 2^{-1}(|\nabla u|^{2}+|\nabla v|^{2})$, we obtain
\begin{align}\label{e3.11kb}
&\langle\varrho_{\mathcal{T}}^{\prime}(u),u-v\rangle \nonumber \\
&=\int_{\Omega}(|\nabla u|^{p(x)-2}\nabla u+\mu_1(x)|\nabla u|^{q(x)-2}\nabla u+\mu_2(x)|\nabla u|^{r(x)-2}\nabla u)\cdot \nabla(u-v) dx \nonumber \\
&=\int_{\Omega}\left\{|\nabla u|^{p(x)}+\mu_1(x)|\nabla u|^{q(x)}+\mu_2(x)|\nabla u|^{r(x)}\right. \nonumber \\
&\left.-(|\nabla u|^{p(x)-2}+\mu_1(x)|\nabla u|^{q(x)-2}+\mu_2(x)|\nabla u|^{r(x)-2})\nabla u \cdot \nabla v\right\}dx \nonumber \\
&\geq  2^{-1}\int_{\Omega}(|\nabla u|^{p(x)-2}+\mu_1(x)|\nabla u|^{q(x)-2}+\mu_2(x)|\nabla u|^{r(x)-2})(|\nabla u|^{2}-|\nabla v|^{2})dx,
\end{align}
and similarly
\begin{align}\label{e3.11kc}
&\langle\varrho_{\mathcal{T}}^{\prime}(v), v-u\rangle \nonumber \\
&=\int_{\Omega}(|\nabla v|^{p(x)-2}\nabla v+\mu_1(x)|\nabla v|^{q(x)-2}\nabla v+\mu_2(x)|\nabla v|^{r(x)-2}\nabla v)\cdot \nabla(v-u) dx \nonumber \\
&=\int_{\Omega}\left\{|\nabla v|^{p(x)}+\mu_1(x)|\nabla v|^{q(x)}+\mu_2(x)|\nabla v|^{r(x)}\right. \nonumber \\
&\left.-(|\nabla v|^{p(x)-2}+\mu_1(x)|\nabla v|^{q(x)-2}+\mu_2(x)|\nabla v|^{r(x)-2})\nabla u \cdot \nabla v\right\}dx \nonumber \\
&\geq  2^{-1}\int_{\Omega}(|\nabla v|^{p(x)-2}+\mu_1(x)|\nabla v|^{q(x)-2}+\mu_2(x)|\nabla v|^{r(x)-2})(|\nabla v|^{2}-|\nabla u|^{2})dx,
\end{align}
Next, we partition $\Omega$ into $\Omega_{\geq}=\{x \in \Omega: |\nabla u|\geq |\nabla v| \}$ and $\Omega_{<}=\{x \in \Omega: |\nabla u|< |\nabla v| \}$. \newline Hence, using (\ref{e3.11kb}), (\ref{e3.11kc}) and $(M)$, we can write
\begin{align}\label{e3.11kd}
I_{\geq}(u):&=\mathcal{M}(\varrho_{\mathcal{T}}(u))\langle\varrho_{\mathcal{H}}^{\prime}(u),u-v\rangle\nonumber \\
&=\mathcal{M}(\varrho_{\mathcal{T}}(u))\int_{\Omega}\left\{|\nabla u|^{p(x)}+\mu_1(x)|\nabla u|^{q(x)}+\mu_2(x)|\nabla u|^{r(x)}\right. \nonumber \\
&\left.-(|\nabla u|^{p(x)-2}+\mu_1(x)|\nabla u|^{q(x)-2}+\mu_2(x)|\nabla u|^{r(x)-2})\nabla u \cdot \nabla v\right\}dx \nonumber \\
& \geq \frac {\mathcal{M}_{\varrho}(u)}{2}\int_{\Omega_{\geq}}(|\nabla u|^{p(x)-2}+\mu_1(x)|\nabla u|^{q(x)-2}+\mu_2(x)|\nabla u|^{r(x)-2})(|\nabla u|^{2}-|\nabla v|^{2})dx\nonumber\\
&- \frac {\mathcal{M}(\varrho_{\mathcal{T}}(v))}{2}\int_{\Omega_{\geq}}(|\nabla v|^{p(x)-2}+\mu_1(x)|\nabla v|^{q(x)-2}+\mu_2(x)|\nabla v|^{r(x)-2})(|\nabla u|^{2}-|\nabla v|^{2})dx\nonumber\\
& \geq \frac {\mathcal{M}(\varrho_{\mathcal{T}}(v))}{2}\int_{\Omega_{\geq}}\left\{(|\nabla u|^{p(x)-2}+\mu_1(x)|\nabla u|^{q(x)-2}+\mu_2(x)|\nabla u|^{r(x)-2})\right.\nonumber \\
&\left.-(|\nabla v|^{p(x)-2}+\mu_1(x)|\nabla v|^{q(x)-2}+\mu_2(x)|\nabla v|^{r(x)-2})\right\} (|\nabla u|^{2}-|\nabla v|^{2})dx\nonumber\\
& \geq \frac {m_0}{2}\int_{\Omega_{\geq}}\left\{(|\nabla u|^{p(x)-2}+\mu_1(x)|\nabla u|^{q(x)-2}+\mu_2(x)|\nabla u|^{r(x)-2})\right.\nonumber\\
&\left.-(|\nabla v|^{p(x)-2}+\mu_1(x)|\nabla v|^{q(x)-2}+\mu_2(x)|\nabla v|^{r(x)-2})\right\}(|\nabla u|^{2}-|\nabla v|^{2})dx\nonumber\\
& \geq 0,
\end{align}
and in a similar fashion
\begin{align}\label{e3.11ke}
I_{<}(v):&=\mathcal{M}(\varrho_{\mathcal{T}}(v))\langle\varrho_{\mathcal{H}}^{\prime}(v),v-u\rangle\nonumber \\
&=\mathcal{M}(\varrho_{\mathcal{T}}(v))\int_{\Omega_{<}}\left\{|\nabla v|^{p(x)}+\mu_1(x)|\nabla v|^{q(x)}+\mu_2(x)|\nabla v|^{r(x)}\right.\nonumber \\
&\left.-(|\nabla v|^{p(x)-2}+\mu_1(x)|\nabla v|^{q(x)-2}+\mu_2(x)|\nabla v|^{r(x)-2}) \nabla u \cdot \nabla v\right\}dx \nonumber \\
&\geq \frac {m_0}{2}\int_{\Omega_{<}}\left\{(|\nabla u|^{p(x)-2}+\mu(x)_1|\nabla v|^{q(x)-2}+\mu(x)_2|\nabla v|^{r(x)-2})\right.\nonumber \\
&\left.-(|\nabla v|^{p(x)-2}+\mu_1(x)|\nabla v|^{q(x)-2}+\mu_2(x)|\nabla v|^{r(x)-2})\right\}(|\nabla u|^{2}-|\nabla v|^{2})dx \nonumber \\
& \geq 0.
\end{align}
Note that
\begin{align}\label{e3.11kf}
\langle \mathcal{H}(u)-\mathcal{H}(v),u-v \rangle =\langle \mathcal{H}(u),u-v \rangle+\langle \mathcal{H}(v),v-u \rangle = I_{\geq}(u)+I_{<}(v)\geq 0.
\end{align}
However, we must discard the case of $\langle \mathcal{H}(u)-\mathcal{H}(v),u-v \rangle =0$ since this would eventually imply that $\partial_{x_i}u=\partial_{x_i}v$ for $i=1,2,...,N$, which would contradict the assumption that $u \neq v $ in $W_0^{1,\mathcal{T}}(\Omega)$. Thus,
\begin{align}\label{e3.11kg}
\langle \mathcal{H}(u)-\mathcal{H}(v),u-v \rangle >0.
\end{align}
\end{proof}
\newpage
\begin{proof}[Proof of Theorem \ref{Thrm:3.1}]
Thanks to Lemmas \ref{Lem:3.2}-\ref{Lem:3.4}, the operator equation (\ref{e3.5c}) has a unique nontrivial solution $u \in W_0^{1,\mathcal{T}}(\Omega)$, which is a nontrivial weak solution to problem (\ref{e1.1}).
\end{proof}

\section{The second problem}

In this section we specify the nonlinearity as $f=f(x,u,\nabla u)$ and show that the problem (\ref{e1.1}) has a nontrivial weak solution in $W_0^{1,\mathcal{T}}(\Omega)$.

\begin{definition}\label{Def:4.1a} A function $u\in W_0^{1,\mathcal{T}}(\Omega)$ is called a weak solution to problem (\ref{e1.1}) if for all test function $\varphi \in W_0^{1,\mathcal{T}}(\Omega)$ it holds
\begin{align}\label{e4.1a}
\mathcal{M}(\varrho_{\mathcal{T}}(u))\langle\varrho^{\prime}_{\mathcal{T}}(u),\varphi\rangle=\int_{\Omega}f(x,u,\nabla u) \varphi dx.
\end{align}
\end{definition}
Let us define the operator $\mathcal{F}:W_0^{1,\mathcal{H}}(\Omega)\rightarrow W_0^{1,\mathcal{T}}(\Omega)^{*}$ as
\begin{equation}\label{e4.3a}
\langle \mathcal{F}(u),\varphi\rangle:=\int_{\Omega}f(x,u,\nabla u)\varphi dx.
\end{equation}
Therefore, we can define the operator $\mathcal{A}:W_0^{1,\mathcal{T}}(\Omega)\rightarrow W_0^{1,\mathcal{T}}(\Omega)^{*}$ by
\begin{equation} \label{e4.4a}
\langle \mathcal{A}(u),\varphi\rangle:=\langle \mathcal{H}(u),\varphi\rangle-\langle \mathcal{F}(u),\varphi\rangle,\,\, \mbox{ for all } u,\varphi \in W_0^{1,\mathcal{T}}(\Omega).
\end{equation}
As with the first problem, we shall solve the operator equation
\begin{equation} \label{e4.5a}
\mathcal{A}u=\mathcal{H}u-\mathcal{F}u=0,
\end{equation}
to obtain a nontrivial weak solution to problem (\ref{e1.1}).\\

For the nonlinearity $f$, we assume:

\begin{itemize}
\item[$(f_{1})$] $f:\Omega \times \mathbb{R} \times \mathbb{R}^{N} \to \mathbb{R}$ is a Carathéodory function such that $f(\cdot,0,0)\neq 0$.
\item[$(f_{2})$] There exists $g \in L^{s^\prime(x)}(\Omega)$ and constants $a_{1},a_{2}>0$ satisfying
\begin{equation*}
|f(x,t,\eta)| \leq g(x)+a_{1}|t|^{s(x)-1} +a_{2}|\eta|^{p(x)\frac{s(x)-1}{s(x)}},
\end{equation*}
for all $(t,\eta) \in \mathbb{R} \times \mathbb{R}^{N}$ and for a.a. $x\in \Omega$, where $s\in C_+(\overline{\Omega}) $ such that $s^\prime(x):= \frac{s(x)}{s(x)-1}$ and $s(x)<p^*(x)$.
\item[$(f_{3})$] There exist $h \in L^{p^\prime(x)}(\Omega)$, $\beta_1 \in L^{\frac {r(x)}{r(x)-\alpha^-}}(\Omega)$ and $\beta_1 \in L^{\frac {r(x)}{r(x)-\alpha(x)}}(\Omega)$ satisfying
\begin{equation*}
\limsup_{|\eta| \to +\infty} \frac{f(x,t,\eta)}{h(x)+\beta_1(x)|t|^{\alpha^--1}+\beta_2(x)|\eta|^{\alpha(x)-1}} \leq \lambda,
\end{equation*}
for all $(t,\eta) \in \mathbb{R} \times \mathbb{R}^{N}$ and for a.a. $x\in \Omega$, where $\alpha \in C_+(\overline{\Omega})$ such that $\alpha^+<p^-$, and $\lambda>0$ is a parameter.
\end{itemize}

\begin{example}
Let's define the function $\hat{f}: \Omega \times \mathbb{R} \times \mathbb{R}^N \to \mathbb{R}$ as follows:
$$
\hat{f}(x,t,\eta) = \hat{g}(x)\sin(t+1)+ \hat{a}_1 |t|^{s(x)-1} e^{-|t|} + \hat{a}_2 |\eta|^{p(x)-1} \ln(1+|\eta|),
$$
where $\hat{g}\in L^{s^*(x)}(\Omega)$, $\hat{a}_1, \hat{a}_2 > 0 $ are positive constants. Then $\hat{f}$ satisfies hypotheses $(f_1)-(f_3)$.
\end{example}

We employ the following result (see, e.g., \cite{zeidler2013nonlinear,papageorgiou2018applied}).
\begin{lemma}\label{Lem:4.1a}
Let $X$ be a reflexive real Banach space. Let $A:X\rightarrow X^{*}$ be a pseudomonotone, bounded, and coercive operator, and $B \in X^{*}$. Then, a solution of the equation $Au=B$ exists.
\end{lemma}

The following is the second main result.

\begin{theorem}\label{Thrm:4.2a} Assume the assumptions $(M)$ and $(f_1)$-$(f_3)$ are satisfied. Then the operator equation (\ref{e4.5a}) has at least one nontrivial solution $u \in W_0^{1,\mathcal{T}}(\Omega)$ which in turn becomes a nontrivial weak solution to problem (\ref{e1.1}).
\end{theorem}

\begin{lemma}\label{Lem:4.3a}
$\mathcal{A}$ is coercive.
\end{lemma}

\begin{proof}
Since the coercivity of $\mathcal{H}$ is shown in Lemma \ref{Lem:3.2}, we  proceed with $\mathcal{F}$.\\
Without loosing generality, we may assume that $|\nabla u|>k_0$ for some constant $k_0\geq 1$, and hence $\|u\|_{1,\mathcal{T},0}>1$. Indeed, if we recall that $\|u\|_{1,\mathcal{T},0}=\||\nabla u|\|_{\mathcal{T}}$, and take into account the monotonicity of convex modular $\rho_\mathcal{T}$ and Remark \ref{Rem:3.4ab}, we can conclude that $\|u\|_{1,\mathcal{T},0}\geq (\rho_\mathcal{T}(\nabla u))^{1/r^+}>1$. Therefore, from $(f_3)$, we have
\begin{equation} \label{e4.6a}
|f(x,u,\nabla u)| \leq \lambda \left(|h|+|\beta_1||u|^{\alpha^--1}+|\beta_2||\nabla u|^{\alpha(x)-1}\right).
\end{equation}
Using H\"{o}lder inequality (see \cite[Proposition 2.3]{avci2019topological}), Proposition \ref{Prop:2.2bb} and invoking the necessary embeddings, we get
\begin{align}\label{e4.7a}
\langle \mathcal{F}(u),u\rangle &\leq \int_{\Omega}\lambda \left(|h||u|+|\beta_1||u|^{\alpha^--1}|u|+|\beta_2||\nabla u|^{\alpha(x)-1}|u|\right)dx\nonumber\\
&\leq \lambda \left(|h|_{p^{\prime}}|u|_{p(x)}+|\beta_1|_{\frac{r(x)}{r(x)-\alpha^-}}||u|^{\alpha^--1}|_{\frac{r(x)}{\alpha^--1}}|u|_{r(x)}\right. \nonumber\\
&\left. +|\beta_2|_{\frac{r(x)}{r(x)-\alpha(x)}}||\nabla u|^{\alpha(x)-1}|_{\frac{r(x)}{\alpha(x)-1}}|u|_{r(x)} \right)\nonumber\\
& \leq \lambda \left(c_1\|u\|^{\alpha^+}_{1,\mathcal{T},0}+c_2\|u\|^{\alpha^-}_{1,\mathcal{T},0}+c_3\|u\|_{1,\mathcal{T},0} \right),
\end{align}
and hence
\begin{equation} \label{e4.8a}
\frac{\langle \mathcal{F}(u),u\rangle}{\|u\|_{1,\mathcal{T},0}} \leq \lambda \left(c_1\|u\|^{\alpha^+-1}_{1,\mathcal{T},0}+c_2\|u\|^{\alpha^--1}_{1,\mathcal{T},0}+c_3 \right).
\end{equation}
From Lemma \ref{Lem:3.2}, we have
\begin{equation}\label{e4.9a}
\frac{\langle\mathcal{H}(u),u\rangle}{\|u\|_{1,\mathcal{T},0}}\geq \frac{m_0}{r^+}\|u\|_{1,\mathcal{T},0}^{p^{-}-1}.
\end{equation}
Then from (\ref{e4.8a}) and (\ref{e4.9a}), we have
\begin{align}\label{ee4.10a}
\frac{\langle \mathcal{A}(u),u\rangle}{\|u\|_{1,\mathcal{T},0}}& \geq \frac{\langle \mathcal{H}(u),u\rangle}{\|u\|_{1,\mathcal{T},0}}-\frac{\langle \mathcal{F}(u),u\rangle}{\|u\|_{1,\mathcal{T},0}}\nonumber\\
& \geq \frac{m_0}{r^+}\|u\|_{1,\mathcal{T},0}^{p^{-}-1}-\lambda \left(c_1\|u\|^{\alpha^+-1}_{1,\mathcal{H},0}+c_2\|u\|^{\alpha^--1}_{1,\mathcal{H},0}+c_3 \right),
\end{align}
hence
\begin{align}\label{e4.11a}
\lim_{\|u\|_{1,\mathcal{T},0} \to \infty}\frac{\langle \mathcal{A}(u),u\rangle}{\|u\|_{1,\mathcal{T},0}}=+\infty.
\end{align}
\end{proof}

\begin{lemma}\label{Lem:4.4a}
$\mathcal{H}$ is continuous and bounded.
\end{lemma}

\begin{proof}
Recall that $\langle \mathcal{H}(u),\varphi \rangle=\mathcal{M}(\varrho_{\mathcal{T}}(u))\langle\varrho^{\prime}_{\mathcal{T}}(u),\varphi\rangle$. By $(M)$, $\mathcal{M}$ is continuous. Additionally, from \cite[Propositions 4.4-4.5]{dai2024regularity},
$\varrho_{\mathcal{T}} \in C^{1}(W_0^{1,\mathcal{T}}(\Omega),\mathbb{R})$ with the Gateaux derivative of $\varrho^{\prime}_{\mathcal{T}}$. Therefore, as a composition function, $\mathcal{M}(\varrho_{\mathcal{T}}(\cdot))$ is continuous. Now, we will show that $\mathcal{H}$ is continuous. To this end, for a sequence $(u_n) \subset W_0^{1,\mathcal{T}}(\Omega)$ assume that $u_n \to u \in W_0^{1,\mathcal{T}}(\Omega)$.  Then using Proposition \ref{Prop:2.8a}, we have
\begin{align}\label{e4.12a}
&|\langle \mathcal{H}(u_n)-\mathcal{H}(u),v \rangle| \nonumber \\
& \leq \int_{\Omega}\bigg|\mathcal{M}_{u_n}|\nabla u_n|^{p(x)-2}\nabla u_n-\mathcal{M}_{u}|\nabla u|^{p(x)-2}\nabla u\bigg| |\nabla v| dx \nonumber \\
& + \int_{\Omega}\mu_1(x)\bigg|\mathcal{M}_{u_n}|\nabla u_n|^{q(x)-2}\nabla u_n-\mathcal{M}_{u}|\nabla u|^{q(x)-2}\nabla u\bigg| |\nabla v| dx \nonumber \\
& + \int_{\Omega}\mu_2(x)\bigg|\mathcal{M}_{u_n}|\nabla u_n|^{r(x)-2}\nabla u_n-\mathcal{M}_{u}|\nabla u|^{r(x)-2}\nabla u\bigg| |\nabla v| dx \nonumber \\
& \leq \int_{\Omega}\left\{||\nabla u_n|^{p(x)-2}\nabla u_n-|\nabla u|^{p(x)-2}\nabla u|\times (K^{u}_{u_n}+\mathcal{M}_{u_n})+K^{u}_{u_n}\right\} |\nabla v| dx \nonumber \\
& + \int_{\Omega}\mu_1(x)\left\{||\nabla u_n|^{q(x)-2}\nabla u_n-|\nabla u|^{q(x)-2}\nabla u|\times (K^{u}_{u_n}+\mathcal{M}_{u_n})+K^{u}_{u_n}\right\} |\nabla v| dx \nonumber \\
& + \int_{\Omega}\mu_2(x)\left\{||\nabla u_n|^{r(x)-2}\nabla u_n-|\nabla u|^{r(x)-2}\nabla u|\times (K^{u}_{u_n}+\mathcal{M}_{u_n})+K^{u}_{u_n}\right\} |\nabla v| dx \nonumber \\
& \leq (K^{u}_{u_n}+\mathcal{M}_{u_n})\int_{\Omega}||\nabla u_n|^{p(x)-2}\nabla u_n-|\nabla u|^{p(x)-2}\nabla u| |\nabla v| dx + K^{u}_{u_n}\int_{\Omega} |\nabla v| dx\nonumber \\
& + (K^{u}_{u_n}+\mathcal{M}_{u_n})\int_{\Omega}\mu_1(x)||\nabla u_n|^{q(x)-2}\nabla u_n-|\nabla u|^{q(x)-2}\nabla u||\nabla v| dx +K^{u}_{u_n}\int_{\Omega} |\nabla v| dx\nonumber \\
& + (K^{u}_{u_n}+\mathcal{M}_{u_n})\int_{\Omega}\mu_2(x)||\nabla u_n|^{r(x)-2}\nabla u_n-|\nabla u|^{r(x)-2}\nabla u||\nabla v| dx +K^{u}_{u_n}\int_{\Omega} |\nabla v| dx,
\end{align}
where we let $\mathcal{M}_{u_n}=\mathcal{M}(\varrho_{\mathcal{T}}(u_n))$, $\mathcal{M}_{u}=\mathcal{M}(\varrho_{\mathcal{T}}(u))$, and $K^{u}_{u_n}=|\mathcal{M}_{u_n}-\mathcal{M}_{u}|$.
Now, if we apply H\"{o}lder's inequality and consider the embeddings $L^{\mathcal{T}}(\Omega)\hookrightarrow L^{q(x)}_{\mu_1}(\Omega)$, $L^{\mathcal{T}}(\Omega)\hookrightarrow L^{r(x)}_{\mu_2}(\Omega)$, and Propositions \ref{Prop:2.2a} and \ref{Prop:2.7a}, it reads
\begin{align}\label{e4.13a}
&|\langle \mathcal{H}(u_n)-\mathcal{H}(u),v \rangle| \nonumber \\
& \leq (K^{u}_{u_n}+\mathcal{M}_{u_n}) \bigg| ||\nabla u_n|^{p(x)-2}\nabla u_n-|\nabla u|^{p(x)-2}\nabla u| \bigg|_{\frac{p(x)}{p(x)-1}}|\nabla v|_{p(x)}+K^{u}_{u_n}|\Omega||\nabla v|_{p(x)}\nonumber \\
&+(K^{u}_{u_n}+\mathcal{M}_{u_n})\bigg|\mu_1(x)^{\frac{q(x)-1}{q(x)}} ||\nabla u_n|^{q(x)-2}\nabla u_n-|\nabla u|^{q(x)-2}\nabla u| \bigg|_{\frac{q(x)}{q(x)-1}}|\mu_1(x)^{\frac{1}{q(x)}}|\nabla v||_{q(x)} \nonumber \\
& +K^{u}_{u_n}|\Omega||\nabla v|_{p(x)}\nonumber \\
&+(K^{u}_{u_n}+\mathcal{M}_{u_n})\bigg|\mu_2(x)^{\frac{r(x)-1}{r(x)}} ||\nabla u_n|^{r(x)-2}\nabla u_n-|\nabla u|^{r(x)-2}\nabla u| \bigg|_{\frac{r(x)}{r(x)-1}}|\mu_2(x)^{\frac{1}{r(x)}}|\nabla v||_{r(x)} \nonumber \\
& +K^{u}_{u_n}|\Omega||\nabla v|_{p(x)},
\end{align}
and therefore
\begin{align}\label{e4.14a}
\|\mathcal{H}(u_n)-\mathcal{H}(u)\|_{W_0^{1,\mathcal{T}}(\Omega)^*}=\sup_{v \in W_0^{1,\mathcal{T}}(\Omega),\|v\|_{1,\mathcal{T},0}\leq 1}|\langle \mathcal{H}(u_n)-\mathcal{H}(u),v \rangle| \to 0.
\end{align}
Hence, $\mathcal{H}$ is continuous on $W_0^{1,\mathcal{T}}(\Omega)$.\\
Note that the result (\ref{e4.14a}) follows from the following reasoning:\\
Since $u_n \to u$ in $W_0^{1,\mathcal{T}}(\Omega)$, the embeddings
\[ L^{\mathcal{T}}(\Omega)\hookrightarrow L^{q(x)}_{\mu_1}(\Omega), L^{\mathcal{T}}(\Omega)\hookrightarrow L^{r(x)}_{\mu_2}(\Omega), W_0^{1,\mathcal{T}}(\Omega) \hookrightarrow L^{\mathcal{T}}(\Omega),  \text{ and }  W_0^{1,\mathcal{T}}(\Omega) \hookrightarrow L^{p(x)}(\Omega) \]
ensure that
\begin{equation}\label{e3.11hk}
\lim_{n \to \infty}\int_{\Omega}|\nabla u_n|^{p(x)}dx=\int_{\Omega}|\nabla u|^{p(x)}dx,
\end{equation}
\begin{equation}\label{e3.11hm}
\lim_{n \to \infty}\int_{\Omega}\mu_1(x)|\nabla u_n|^{q(x)}dx=\int_{\Omega}\mu_1(x)|\nabla u|^{q(x)}dx,
\end{equation}
and
\begin{equation}\label{e3.11hma}
\lim_{n \to \infty}\int_{\Omega}\mu_2(x)|\nabla u_n|^{r(x)}dx=\int_{\Omega}\mu_2(x)|\nabla u|^{r(x)}dx.
\end{equation}
By Vitali’s Theorem (see \cite[Theorem 4.5.4]{bogachev2007measure} or \cite[Theorem 8]{royden2010real}), equations (\ref{e3.11hk})-(\ref{e3.11hma}) imply that $|\nabla u_n| \to |\nabla u|$, $\mu_1(x)^{\frac{1}{q(x)}}|\nabla u_n| \to \mu_1(x)^{\frac{1}{q(x)}}|\nabla u|$, and $\mu_2(x)^{\frac{1}{r(x)}}|\nabla u_n| \to \mu_2(x)^{\frac{1}{r(x)}}|\nabla u|$ in measure in $\Omega$. Moreover, the sequences $\{|\nabla u_n|^{p(x)}\}$, $\{\mu_1(x)|\nabla u_n|^{q(x)}\}$, and $\{\mu_2(x)|\nabla u_n|^{r(x)}\}$ are uniformly integrable over $\Omega$.
Now, consider the inequalities
\begin{equation}\label{e3.11hn}
||\nabla u_n|^{p(x)-2}\nabla u_n-|\nabla u|^{p(x)-2}\nabla u|^{\frac{p(x)}{p(x)-1}} \leq 2^{\frac{p^+}{p^--1}-1}(|\nabla u_n|^{p(x)}+|\nabla u|^{p(x)}),
\end{equation}
and
\begin{equation}\label{e3.11hnm}
\mu_1(x)||\nabla u_n|^{q(x)-2}\nabla u_n-|\nabla u|^{q(x)-2}\nabla u|^{\frac{q(x)}{q(x)-1}} \leq 2^{\frac{q^+}{q^--1}-1}\mu_1(x)(|\nabla u_n|^{q(x)}+|\nabla u|^{q(x)}),
\end{equation}
and
\begin{equation}\label{e3.11hnmn}
\mu_2(x)||\nabla u_n|^{r(x)-2}\nabla u_n-|\nabla u|^{r(x)-2}\nabla u|^{\frac{r(x)}{r(x)-1}} \leq 2^{\frac{r^+}{r^--1}-1}\mu_2(x)(|\nabla u_n|^{r(x)}+|\nabla u|^{r(x)}).
\end{equation}
As a consequence, the families
\begin{equation}\label{e3.12hnmn}
\bigg\{||\nabla u_n|^{p(x)-2}\nabla u_n-|\nabla u|^{p(x)-2}\nabla u|^{\frac{p(x)}{p(x)-1}}\bigg\},
\end{equation}
\begin{equation}\label{e3.13hnmn}
\bigg\{\mu_1(x)||\nabla u_n|^{q(x)-2}\nabla u_n-|\nabla u|^{q(x)-2}\nabla u|^{\frac{q(x)}{q(x)-1}}\bigg\},
\end{equation}
and
\begin{equation}\label{e3.14hnmn}
\bigg\{\mu_2(x)||\nabla u_n|^{r(x)-2}\nabla u_n-|\nabla u|^{r(x)-2}\nabla u|^{\frac{r(x)}{r(x)-1}}\bigg\}
\end{equation}
are uniformly integrable over $\Omega$. Applying Vitali’s Theorem again, we deduce that $|\nabla u|^{p(x)-2}\nabla u$, $\mu_1(x)|\nabla u|^{q(x)-2}\nabla u$ and $\mu_2(x)|\nabla u|^{r(x)-2}\nabla u$ are integrable, and
\begin{equation}\label{e3.11hnk}
 \bigg| ||\nabla u_n|^{p(x)-2}\nabla u_n-|\nabla u|^{p(x)-2}\nabla u| \bigg|_{\frac{p(x)}{p(x)-1}} \to 0,
\end{equation}
\begin{equation}\label{e3.11hnsn}
\bigg|\mu_1(x)^{\frac{q(x)-1}{q(x)}} ||\nabla u_n|^{q(x)-2}\nabla u_n-|\nabla u|^{q(x)-2}\nabla u| \bigg|_{\frac{q(x)}{q(x)-1}} \to 0,
\end{equation}
and
\begin{equation}\label{e3.11hns}
\bigg|\mu_2(x)^{\frac{r(x)-1}{r(x)}} ||\nabla u_n|^{r(x)-2}\nabla u_n-|\nabla u|^{r(x)-2}\nabla u| \bigg|_{\frac{r(x)}{r(x)-1}} \to 0.
\end{equation}
Finally, by assumption $(M)$ and Proposition \ref{Prop:2.7}, we have \newline $K^{u}_{u_n}=|\mathcal{M}_{u_n}-\mathcal{M}_{u}| \to 0$, and $\mathcal{M}_{u_n} \to \mathcal{M}_{u} \in [m_0, \infty)$ as $n \to \infty$. Therefore, the result (\ref{e4.14a}) follows. \\
Now, we verify that $\mathcal{H}$ is bounded. We argue similarly to \cite[Propositions 4.5]{dai2024regularity}.\\
Letting $u, v \in W_0^{1,\mathcal{T}}(\Omega)\setminus \{0\}$, using $(M)$, Remark \ref{Rem:3.4ab} and Young’s inequality, we obtain
\begin{align}\label{e4.12a}
&\min\bigg\{\frac{1}{\|\nabla v\|^{p^--1}_{\mathcal{T}}},\frac{1}{\|\nabla v\|^{r^+-1}_{\mathcal{T}}}\bigg\}\bigg\langle \mathcal{H}(u), \frac{v}{\|\nabla v\|_{\mathcal{T}}}\bigg\rangle \nonumber\\
&\leq \kappa \rho^{\gamma-1}_{\mathcal{T}}(\nabla u)\int_{\Omega}\left(\bigg|\frac{\nabla u}{\|\nabla u \|_{\mathcal{T}}}\bigg|^{p(x)-1}\frac{|\nabla v|}{\|\nabla v\|_{\mathcal{T}}} \right. \nonumber\\
&\left. +\mu_1(x)^{\frac{q(x)-1}{q(x)}}\bigg|\frac{\nabla u}{\|\nabla u \|_{\mathcal{T}}}\bigg|^{q(x)-1}\mu_1(x)^{\frac{1}{q(x)}}\frac{|\nabla v|}{\|\nabla v\|_{\mathcal{T}}}\right. \nonumber\\
&\left. +\mu_2(x)^{\frac{r(x)-1}{r(x)}}\bigg|\frac{\nabla u}{\|\nabla u \|_{\mathcal{T}}}\bigg|^{r(x)-1}\mu_2(x)^{\frac{1}{r(x)}}\frac{|\nabla v|}{\|\nabla v\|_{\mathcal{T}}}\right) dx \nonumber\\
&\leq \kappa \rho^{\gamma-1}_{\mathcal{T}}(\nabla u)\int_{\Omega}\left(\frac{p^+-1}{p^-}\bigg|\frac{\nabla u}{\|\nabla u \|_{\mathcal{T}}}\bigg|^{p(x)}+\frac{1}{p^-}\frac{\nabla v}{\|\nabla v \|_{\mathcal{T}}}\bigg|^{p(x)} \right. \nonumber\\
&\left. +\frac{\mu_1(x)(q^+-1)}{q^-}\bigg|\frac{\nabla u}{\|\nabla u \|_{\mathcal{T}}}\bigg|^{q(x)}\frac{\mu_1(x)}{q^-}\bigg|\frac{\nabla v}{\|\nabla v \|_{\mathcal{T}}}\bigg|^{q(x)}\right. \nonumber\\
&\left. +\frac{\mu_2(x)(r^+-1)}{r^-}\bigg|\frac{\nabla u}{\|\nabla u \|_{\mathcal{T}}}\bigg|^{r(x)}\frac{\mu_2(x)}{r^-}\bigg|\frac{\nabla v}{\|\nabla v \|_{\mathcal{T}}}\bigg|^{r(x)}\right) dx \nonumber\\
& \leq \kappa \rho^{\gamma-1}_{\mathcal{T}}(\nabla u)\left(\frac{r^+-1}{p^-}\rho_{\mathcal{T}}\left(\frac{\nabla u}{\|\nabla u \|_{\mathcal{T}}}\right)+\frac{1}{p^-}\rho_{\mathcal{T}}\left(\frac{\nabla v}{\|\nabla v \|_{\mathcal{T}}}\right) \right)\nonumber\\
& \leq \frac{\kappa r^+}{p^-} \|\nabla u\|^{(\gamma-1)\tau}_{\mathcal{T}},
\end{align}
where $\tau =\max\{r^{+}, p^{-}\}$.  Therefore,
\begin{align}\label{e4.14ba}
\|\mathcal{H}(u)\|_{W_0^{1,\mathcal{T}}(\Omega)^*}&=\sup_{v \in W_0^{1,\mathcal{T}}(\Omega)\setminus\{0\}} \frac{\langle\mathcal{H}(u),v \rangle}{\|\nabla v \|_{\mathcal{T}}} \leq \frac{\kappa r^+}{p^-} \|\nabla u\|^{(\gamma-1)\tau}_{\mathcal{T}}\max\{\|\nabla u\|^{p^--1}_{\mathcal{T}},\|\nabla u\|^{r^+-1}_{\mathcal{T}}\},
\end{align}
which concludes that $\mathcal{H}$ is bounded.
\end{proof}

\begin{lemma}\label{Lem:4.5a}
$\mathcal{F}$ is continuous and bounded.
\end{lemma}

\begin{proof}
Define the operator $\mathcal{B}_f:W_0^{1,\mathcal{T}}(\Omega)\rightarrow L^{s^{\prime}(x)}(\Omega)$ by
\begin{equation} \label{e3.5}
\mathcal{B}_f(u)=f(x,u,\nabla u).
\end{equation}
First we show that $\mathcal{B}_f$ is bounded in $L^{s^{\prime}(x)}(\Omega)$. Then, for any $u$ satisfying $\|u\|_{1,\mathcal{T},0}\leq1$, employing assumption $(f_2)$ along with the embeddings $L^{\mathcal{T}}(\Omega)\hookrightarrow L^{s(x)}(\Omega)$ and $W_0^{1,\mathcal{T}}(\Omega)\hookrightarrow L^{\mathcal{T}}(\Omega)$, we obtain
\begin{align}\label{e3.6}
\int_{\Omega} |\mathcal{B}_f(u)|^{s^\prime(x)} dx &=\int_{\Omega} |f(x,u,\nabla u)|^{s^\prime(x)} dx\nonumber\\
&\leq \int_{\Omega}  |g(x)+a_{1}|u|^{s(x)-1} +a_{2}|\nabla u|^{p(x)\frac{s(x)-1}{s(x)}}|^{s^\prime(x)}dx  \nonumber\\
&\leq c\int_{\Omega} \left(|g(x)|^{s^\prime(x)}+|u|^{s(x)}+|\nabla u|^{p(x)}\right)dx \nonumber\\
&\leq c\int_{\Omega} \left(|g(x)|^{s^\prime(x)}+|u|^{s(x)}+(|\nabla u|^{p(x)}+\mu_1(x)|\nabla u|^{q(x)}+\mu_2(x)|\nabla u|^{r(x)})\right)dx \nonumber\\
&\leq c\|u\|_{1,\mathcal{T},0}.
\end{align}
Now, let $u_n \to u$ in $W_0^{1,\mathcal{T}}(\Omega)$. Thus, we have $\nabla u_n \to \nabla u$ in $L^{s(x)}(\Omega)^{N}$. This ensures the existence of a subsequence, still denoted by $(u_n)$, and functions $\omega_1(x) \in L^{s(x)}(\Omega)$ and $\omega_2(x) \in L^{s(x)}(\Omega)^{N}$ satisfying:
\begin{itemize}
    \item $u_n(x) \to u(x)$ and $\nabla u_n(x) \to \nabla u(x)$ almost everywhere in $\Omega$,
    \item $|u_n(x)|\leq \omega_1(x)$ and $|\nabla u_n(x)|\leq |\omega_2(x)|$ almost everywhere in $\Omega$ for all $n$.
\end{itemize}
By assumption $(f_1)$, the function $f$ is continuous in its second and third arguments, leading to
\begin{equation} \label{e3.7}
f(x,u_n(x),\nabla u_n(x))\to f(x,u(x),\nabla u(x)) \text{ a.e. in } \Omega \text{ as } n\to \infty.
\end{equation}
Moreover, utilizing the previous bounds and assumption $(f_2)$, we obtain
\begin{align} \label{e3.8}
|f(x,u_n(x),\nabla u_n(x))|& \leq g(x)+a_{1}|\omega_1(x)|^{s(x)-1}+a_{2}|\omega_2(x)|^{p(x)\frac{s(x)-1}{s(x)}}.
\end{align}
By Hölder's inequality and Proposition \ref{Prop:2.2bb}, we establish
\begin{equation}\label{e3.9}
\int_{\Omega}a_{1}|\omega_1|^{s(x)-1}dx \leq c ||\omega_1|^{s(x)-1}|_{\frac{s(x)}{s(x)-1}}|\mathbf{1}|_{s(x)} \leq c |\omega_1|^{s^+-1}_{s(x)},
\end{equation}
and
\begin{equation}\label{e3.10}
\int_{\Omega}a_{2}|\omega_2|^{p(x)\frac{s(x)-1}{s(x)}}dx \leq c||\omega_2|^{p(x)\frac{s(x)-1}{s(x)}}|_{\frac{s(x)}{s(x)-1}}|\mathbf{1}|_{s(x)} \leq c |\omega_2|^{p^+}_{p(x)}.
\end{equation}
Since the embeddings $L^{\mathcal{T}}(\Omega)\hookrightarrow L^{s(x)}(\Omega)$, $W_0^{1,\mathcal{T}}(\Omega)\hookrightarrow W_0^{1,p(x)}(\Omega)$, and $W_0^{1,\mathcal{T}}(\Omega)\hookrightarrow L^{\mathcal{T}}(\Omega)$ ensure that the right-hand side of (\ref{e3.8}) is integrable, we apply the Lebesgue dominated convergence theorem (see, e.g., \cite{royden2010real}) together with (\ref{e3.7}) to conclude that
\begin{equation}\label{e3.11}
f(x,u_n,\nabla u_n) \to f(x,u,\nabla u) \text{ in } L^{1}(\Omega).
\end{equation}
Thus, we obtain
\begin{align}\label{e3.12}
\lim_{n \to \infty}\int_{\Omega} |\mathcal{B}_f(u_n)-\mathcal{B}_f(u)|^{s^\prime(x)} dx=0,
\end{align}
which, by Proposition \ref{Prop:2.2}, implies that $\mathcal{B}_f$ is continuous in $L^{s^{\prime}(x)}(\Omega)$.
Finally, since the embedding operator $i^*:L^{s^{\prime}(x)}(\Omega) \to W_0^{1,\mathcal{T}}(\Omega)^*$ is continuous, defining $\mathcal{F}= i^*\circ \mathcal{B}_f$ gives $\mathcal{F}:W_0^{1,\mathcal{T}}(\Omega)\rightarrow W_0^{1,\mathcal{T}}(\Omega)^{*}$ as a continuous and bounded operator.
\end{proof}
\begin{lemma}\label{Lem:4.6a}
$\mathcal{A}$ is pseudomonotone, i.e. for a sequence $(u_{n})\subset W_0^{1,\mathcal{T}}(\Omega)$,
\begin{equation}\label{e4.7ab}
u_{n} \rightharpoonup u \in W_0^{1,\mathcal{T}}(\Omega)
\end{equation}
and
\begin{equation}\label{e4.8ab}
\limsup_{n\rightarrow\infty}\langle \mathcal{A}(u_{n}),u_{n}-u\rangle\leq 0
\end{equation}
imply
\begin{equation}\label{e4.9ab}
\liminf_{n\rightarrow\infty}\langle \mathcal{A}(u_{n}),u_{n}-v\rangle\geq \langle \mathcal{A}(u),u-v\rangle, \quad  \forall v \in W_0^{1,\mathcal{T}}(\Omega).
\end{equation}
\end{lemma}

\begin{proof}
As a consequence of Lemmas \ref{Lem:4.4a} and \ref{Lem:4.5a}, $\mathcal{A}$ is continuous and bounded. Since $\mathcal{A}$ is bounded, we use an equivalent definition of pseudomonotonicity (see, e.g.,\cite{crespo2022new}) as follows:\\
\[ u_{n} \rightharpoonup u \in W_0^{1,\mathcal{T}}(\Omega) \text{  and  } \limsup_{n\rightarrow\infty}\langle \mathcal{A}(u_{n}),u_{n}-u\rangle\leq 0\]
\text{imply}
\[ \mathcal{A}(u_{n}) \rightharpoonup \mathcal{A}(u)  \text{  and  }  \langle \mathcal{A}(u_{n}),u_{n}\rangle \rightarrow \langle \mathcal{A}(u),u\rangle.\]
To this end, let $(u_{n})\subset W_0^{1,\mathcal{T}}(\Omega)$ with
\begin{equation}\label{e4.9abc}
u_{n} \rightharpoonup u \in W_0^{1,\mathcal{T}}(\Omega) \text{ and } \limsup_{n\rightarrow\infty}\langle \mathcal{A}(u_{n}),u_{n}-u\rangle\leq 0.
\end{equation}
By the weak convergence of $(u_{n})$ in $W_0^{1,\mathcal{T}}(\Omega)$, $(u_{n})$ is $\|\cdot\|_{1,\mathcal{T},0}$-bounded. Thus, considering the compact embedding $W_0^{1,\mathcal{T}}(\Omega)\hookrightarrow L^{s(x)}(\Omega)$, and the boundedness of $\mathcal{B}_f$, it reads
\begin{align}\label{e4.10ab}
\bigg| \int_{\Omega}f(x,u_n,\nabla u_n)(u_n-u)dx \bigg|&\leq c |\mathcal{B}_f(u_n)|_{\frac{s(x)-1}{s(x)}}|u_n-u|_{s(x)}\nonumber\\
& \leq c \sup_{n \in \mathbb{N}} |\mathcal{B}_f(u_n)|_{\frac{s(x)-1}{s(x)}}|u_n-u|_{s(x)} \to 0 \quad \text{as } n \to \infty,
\end{align}
which means
\begin{equation}\label{e4.11ab}
\lim_{n \to \infty}\int_{\Omega}f(x,u_n,\nabla u_n)(u_n-u)dx=0.
\end{equation}
Furthermore, by taking the limit in the weak formulation of (\ref{e4.1a}) while substituting $u$ with $u_n$ and $\varphi$ with $u_n - u$, and considering $(M)$, we arrive at
\begin{equation}\label{e4.12ab}
\limsup_{n\rightarrow\infty}\langle \mathcal{H}(u_{n}),u_{n}-u\rangle=\limsup_{n\rightarrow\infty}\langle \mathcal{A}(u_{n}),u_{n}-u\rangle\leq 0.
\end{equation}
This, by Proposition \ref{Prop:2.7}, means that $\mathcal{H}$ satisfies the $(S_{+})$-property. Moreover, by (\ref{e4.9abc}) and (\ref{e4.12ab}), it reads $u_{n} \to u$ in $W_0^{1,\mathcal{T}}(\Omega)$. Finally, considering that $\mathcal{A}$ is continuous and bounded, we obtain $\mathcal{A}(u_{n}) \to \mathcal{A}(u)$ in $W_0^{1,\mathcal{T}}(\Omega)^*$, from which we conclude that $\mathcal{A}$ is pseudomonotone.
\end{proof}

\begin{proof}[Proof of Theorem \ref{Thrm:4.2a}]
Since $\mathcal{A}$ is a pseudomonotone, bounded, and coercive operator, it is surjective. This result ensures the existence of a function $u \in W_0^{1,\mathcal{T}}(\Omega)$ such that  $\mathcal{A}u=\mathcal{H}u-\mathcal{F}u=0$. On the other hand, by the definition of $\mathcal{A}$ and the assumptions $(M)$, $(f_1)$, $u$ is a nontrivial weak solution of problem (\ref{e1.1}).
\end{proof}

\section*{Conflict of Interest}
The author declared that he has no conflict of interest.

\section*{Data Availability}
No data is used to conduct this research.

\section*{Funding}
This work was supported by Athabasca University Research Incentive Account [140111 RIA].

\section*{ORCID}
https://orcid.org/0000-0002-6001-627X

\bibliographystyle{tfnlm}
\bibliography{references}

\begin{thebibliography}{10}
\providecommand{\url}[1]{\normalfont{#1}}
\providecommand{\urlprefix}{Available from: }

\bibitem{de2019regularity}
De~Filippis~C, Oh~J. Regularity for multi-phase variational problems. Journal
  of Differential Equations. 2019;\hspace{0pt}267(3):1631--1670.

\bibitem{vetro2024priori}
Vetro~F. A priori upper bounds and extremal weak solutions for multi-phase
  problems with variable exponents. Discrete and Continuous Dynamical
  Systems-S. 2024;\hspace{0pt}:0--0.

\bibitem{dai2024regularity}
Dai~G, Vetro~F. Regularity and uniqueness to multi-phase problem with variable
  exponent. arXiv preprint arXiv:240714123. 2024;\hspace{0pt}.

\bibitem{zhikov1987averaging}
Zhikov~VV. Averaging of functionals of the calculus of variations and
  elasticity theory. Mathematics of the USSR-Izvestiya.
  1987;\hspace{0pt}29(1):33.

\bibitem{baroni2015harnack}
Baroni~P, Colombo~M, Mingione~G. Harnack inequalities for double phase
  functionals. Nonlinear Analysis: Theory, Methods \& Applications.
  2015;\hspace{0pt}121:206--222.

\bibitem{baroni2018regularity}
Baroni~P, Colombo~M, Mingione~G. Regularity for general functionals with double
  phase. Calculus of Variations and Partial Differential Equations.
  2018;\hspace{0pt}57:1--48.

\bibitem{colombo2015bounded}
Colombo~M, Mingione~G, et~al. Bounded minimisers of double phase variational
  integrals. Arch Ration Mech Anal. 2015;\hspace{0pt}218(1):219--273.

\bibitem{colombo2015regularity}
Colombo~M, Mingione~G. Regularity for double phase variational problems.
  Archive for Rational Mechanics and Analysis. 2015;\hspace{0pt}215:443--496.

\bibitem{marcellini1991regularity}
Marcellini~P. Regularity and existence of solutions of elliptic equations with
  p, q-growth conditions. Journal of Differential Equations.
  1991;\hspace{0pt}90(1):1--30.

\bibitem{marcellini1989regularity}
Marcellini~P. Regularity of minimizers of integrals of the calculus of
  variations with non standard growth conditions. Archive for Rational
  Mechanics and Analysis. 1989;\hspace{0pt}105:267--284.

\bibitem{vetro2025multiplicity}
Vetro~F. Multiplicity of solutions for a kirchhoff multi-phase problem with
  variable exponents. Acta Applicandae Mathematicae. 2025;\hspace{0pt}195(1):5.

\bibitem{cruz2013variable}
Cruz-Uribe~DV, Fiorenza~A. Variable lebesgue spaces: Foundations and harmonic
  analysis. Springer Science \& Business Media; 2013.

\bibitem{diening2011lebesgue}
Diening~L, Harjulehto~P, H{\"a}st{\"o}~P, et~al. Lebesgue and sobolev spaces
  with variable exponents. Springer; 2011.

\bibitem{edmunds2000sobolev}
Edmunds~D, R{\'a}kosn{\'\i}k~J. Sobolev embeddings with variable exponent.
  Studia Mathematica. 2000;\hspace{0pt}3(143):267--293.

\bibitem{fan2001spaces}
Fan~X, Zhao~D. On the spaces {$L^{p(x)}(\Omega)$} and {$W^{m,p(x)}(\Omega)$}.
  Journal of mathematical analysis and applications.
  2001;\hspace{0pt}263(2):424--446.

\bibitem{radulescu2015partial}
Radulescu~VD, Repovs~DD. Partial differential equations with variable
  exponents: variational methods and qualitative analysis. Vol.~9. CRC press;
  2015.

\bibitem{zeidler2013nonlinear}
Zeidler~E. Nonlinear functional analysis and its applications: Ii/b: Nonlinear
  monotone operators. Springer Science \& Business Media; 2013.

\bibitem{browder1963nonlinear}
Browder~FE. Nonlinear elliptic boundary value problems. Bulletin of the
  American Mathematical Society. 1963;\hspace{0pt}69(6):862--874.

\bibitem{minty1963monotonicity}
Minty~GJ. On a monotonicity method for the solution of nonlinear equations in
  banach spaces. Proceedings of the National Academy of Sciences.
  1963;\hspace{0pt}50(6):1038--1041.

\bibitem{chipot2009elliptic}
Chipot~M. Elliptic equations: an introductory course. Springer Science \&
  Business Media; 2009.

\bibitem{massar2024existence}
Massar~M. Existence results for an anisotropic variable exponent kirchhoff-type
  problem. Complex Variables and Elliptic Equations.
  2024;\hspace{0pt}69(2):234--251.

\bibitem{papageorgiou2018applied}
Papageorgiou~NS, Winkert~P. Applied nonlinear functional analysis: An
  introduction. Walter de Gruyter GmbH \& Co KG; 2018.

\bibitem{avci2019topological}
Avci~M. A topological result for a class of anisotropic difference equations.
  Annals of the University of Craiova-Mathematics and Computer Science Series.
  2019;\hspace{0pt}46(2):328--343.

\bibitem{bogachev2007measure}
Bogachev~V. Measure theory springer ; 2007.

\bibitem{royden2010real}
Royden~H, Fitzpatrick~PM. Real analysis. China Machine Press; 2010.

\bibitem{crespo2022new}
Crespo-Blanco~{\'A}, Gasi{\'n}ski~L, Harjulehto~P, et~al. A new class of double
  phase variable exponent problems: Existence and uniqueness. Journal of
  Differential Equations. 2022;\hspace{0pt}323:182--228.

\end{thebibliography}

\end{document}